\documentclass[12pt]{amsart}

\usepackage{xcolor}
\definecolor{cite}{RGB}{44,123,182}
\definecolor{ref}{RGB}{215,25,28}
  \usepackage{url}
  \usepackage[pdftex,
              pdfauthor={Bronson Lim, Franco Rota},
              pdftitle={Riemann-Roch coefficients for Kleinian orbisurfaces},
              pdfsubject={Algebraic geometry, orbifold Riemann Roch, Kleinian singularity, stacky surfaces},
              pdfkeywords={orbifold riemann roch coefficient franco rota bronson lim},
              colorlinks=true,
              citecolor=cite,
              linkcolor=ref 
              ]{hyperref}

\usepackage {crimson}

  \usepackage{amsfonts,graphics,amsmath,amsthm,amscd, amssymb,amsmath,latexsym,bbm,bm,mathtools,mathrsfs}
  \usepackage{flafter}
  \usepackage{verbatim,lipsum}
  \usepackage{enumerate}
  \usepackage{cases}

\usepackage[top=30mm,bottom=30mm,left=30mm,right=30mm]{geometry}

  \usepackage[english]{babel}
\usepackage[utf8]{inputenc}
  \usepackage{tikz}
  \usetikzlibrary{cd,matrix,arrows,decorations.pathmorphing}

\renewcommand{\arraystretch}{1.1}

\theoremstyle{plain}
\newtheorem{thm}{Theorem}[section]

\newtheorem{lemma}[thm]{Lemma}

\newtheorem{proposition}[thm]{Proposition}

\newtheorem*{thm*}{Theorem}
\newtheorem*{corollary*}{Corollary}
\newtheorem*{lemma*}{Lemma}
\newtheorem*{ld*}{Lemma/Definition}
\newtheorem*{proposition*}{Proposition}
\newtheorem*{assumption*}{Assumption}

\theoremstyle{definition}
\newtheorem{definition}[thm]{Definition}
\newtheorem{example}[thm]{Example}

\newtheorem*{definition*}{Definition}
\newtheorem*{example*}{Example}
\newtheorem*{xca*}{Exercise}
\newtheorem*{claim*}{Claim}
\newtheorem*{fact*}{Fact}
\newtheorem*{notation*}{Notation}
\newtheorem*{construction*}{Construction}
\newtheorem*{ack*}{Acknowledgements}
\newtheorem*{question*}{Question}
\newtheorem*{problem*}{Problem}
\newtheorem*{conjecture*}{Conjecture}

\theoremstyle{remark}
\newtheorem{remark}[thm]{Remark}

\newcommand{\C}{\mathbb{C}}
\newcommand{\Z}{\mathbb{Z}}
\newcommand{\Q}{\mathbb{Q}}

\DeclareMathOperator{\GL}{GL}
\DeclareMathOperator{\SL}{SL}

\DeclareMathOperator{\dL}{\mathbf{L}}

\DeclareMathOperator{\textch}{ch}												

\newcommand{\ch}[1]{\textch_{{#1}}}

\newcommand{\sX}{\mathcal{X}}       
\newcommand{\sY}{\mathcal{Y}}

\newcommand{\sO}{\mathcal{O}}

\newcommand{\sS}{\mathcal{S}}

\newcommand{\sF}{\mathcal{F}}

\newcommand{\st}{\,|\,}					

\usepackage{leftidx}						

\newcommand{\AAA}[1]{\textcolor{red}{#1}}

\DeclarePairedDelimiter{\set}{\lbrace}{\rbrace}

\DeclarePairedDelimiter{\pair}{\langle}{\rangle}

\DeclarePairedDelimiter{\abs}{\lvert}{\rvert}

\begin{document}

\title[Riemann-Roch coefficients for Kleinian orbisurfaces]{Riemann-Roch coefficients for Kleinian orbisurfaces}

\author[B. Lim]{Bronson Lim}
\address{BL: Department of Mathematics \\ California State University San Bernardino \\ San Bernardino,
  CA 92407, USA} 
\email{BLim@csusb.edu}

\author[F. Rota]{Franco Rota}
\address{FR: School of Mathematics and Statistics, University of Glasgow, Glasgow, UK}
\email{franco.rota@glasgow.ac.uk}

\thanks{On behalf of all authors, the corresponding author states that there is no conflict of interest.}



\begin{abstract}
	Suppose \(\mathcal{S}\) is a smooth, proper, and tame Deligne-Mumford stack. To\"en's Grothendieck-Riemann-Roch theorem requires correction terms, involving components of the inertia stack, to the standard formula for schemes. We give a brief overview of To\"en's Grothendieck-Riemann-Roch theorem, and explicitly compute the correction terms in the case of an orbifold surface with stabilizers of types ADE.
\end{abstract}

\maketitle




\section{Introduction}

Let \(\sS\) be a \textit{Kleinian orbisurface} over an algebraically closed field \(\mathbf{k}\). That is, \(\sS\) is a smooth, proper, and tame Deligne-Mumford surface with isolated stacky locus, ADE stabilizers, and projective coarse moduli. The classical Riemann-Roch theorem fails for orbisurfaces as it fails to capture contributions from the stacky locus. The To\"en-Hirzebruch-Riemann-Roch theorem provides corrections terms to the classical formula for \(\sS\) coming from the inertia stack \cite{Toe99}. These correction terms arise by pulling back to the inertia stack to determine the contributions from the twisted sectors. 

The purpose of this paper is to determine the correction terms for \(\sS\). This has already been accomplished when \(\sS\) has a stacky point of type \(A\), \cite[Section 3.3]{Lie11}. In \cite[Appendix A]{CT19}, the authors compute the correction term for the sheaf $\sO_\sS$ in all $ADE$ cases. Moreover, they relate the correction term to the exceptional divisor of the minimal resolution of the coarse moduli \(S\) of \(\sS\). This is a manifestation of the famous Bridgeland-King-Reid theorem asserting a derived equivalence between the derived categories of \(\sS\) and \(S\) \cite{BKR01}. 

The remaining cases are the binary tetrahedral (\(E_6\)), binary octahedral (\(E_7\)), binary icosahedral (\(E_8\)), and the binary dihedral groups (\(D\)). 
In Section \ref{sec:prelims}, we give an overview of To\"en's Riemann-Roch theorem for stacks. Section \ref{sec:orbisurf} introduces Kleinian orbisurfaces and spells out the Riemann-Roch theorem for this case. In Proposition \ref{thm_RR_coeff}, we recall a formula for the correction terms for a general ADE singularity, which relies on coefficients determined solely by the character table. 
In Section \ref{sec_Correction_terms}, we compute explicitly these Riemann-Roch coefficients in each of the ADE cases. In a different language, similar correction terms have been computed in \cite[Sec. 5]{Lan00}; however, the present formulation is more natural when studying orbisurfaces from the stacky point of view. 

\subsection*{Notation and conventions}
For a smooth quasi-projective scheme $X$ over $\C$, we denote by $K(X)$ the Grothendieck group of coherent sheaves on $X$, and by $H^*(X)$ the singular cohomology with rational coefficients of its associated analytic space. We assume all algebraic stacks to be Deligne-Mumford, which implies that stabilizers are finite groups.
We use Vistoli's definition of Chern and Todd classes for Deligne-Mumford stacks \cite{Vistoli}, and denote by $K(\sX)$ and $H^*(\sX)$ the Grothendieck group and the (rational) singular homology of a smooth Deligne-Mumford stack $\sX$ as described in \cite{AGV08}.

\section{To\"en's Grothendieck-Riemann-Roch theorem for Deligne-Mumford stacks}
\label{sec:prelims}

In this section, we give an overview of To\"en's Riemann-Roch theorem for stacks \cite{Toe99}. The theorem holds in arbitrary characteristic, but for ease of exposition we work over the field of complex numbers\footnote{The main technical subtlety when working in positive characteristic is to make sure that the Chern character takes values in an appropriate cohomology theory. In \cite{Toe99}, this is 
\'etale cohomology of an algebraic stack.}. We direct the interested reader to other accounts of this result such as \cite[Appendix A]{BCY12}, \cite[Appendix A]{Tse10}, and to Edidin's equivariant Riemann-Roch formulation \cite{Edi13}.

Our first step is to recall the statements for schemes (see for example \cite[Chapter 15]{Ful84}). Let $X$ be a smooth projective scheme, and $E$ a perfect complex of sheaves on $X$. Denote by $K(X)$ the Grothendieck group of coherent sheaves of $X$, and by $H^*(X)$ its singular cohomology. Then there is a linear map\footnote{Here and from now on we use the same notation for $E$ and its class in the Grothendieck group, unless confusion can arise.}
\begin{equation}
    \label{eq:RRtau_classic}
        \tau_X \colon  K(X) \longrightarrow  H^*(X) \qquad \qquad  E \longmapsto \ch{}(E)\cdot \mathrm{Td}(X),
\end{equation}
where $\ch{}$ denotes the Chern character and $\mathrm{Td}$ denotes the Todd class. 
The Grothendieck-Riemann-Roch theorem states that $\tau$ is functorial with respect to proper push forwards, i.e. if $f\colon X \to Y$ is a proper map of smooth quasi-projective schemes, and $E$ is a perfect complex of sheaves on $X$, then $\tau_Y(f_* E) = f_*(\tau_X(E))$. The special case of $Y=\mathrm{pt}$ yields the Hirzebruch-Riemann-Roch theorem, which asserts 
\[ \chi(E) = \int\limits_X \ch{}(E)\cdot \mathrm{Td}(X).  \]

For $\sX$ a Deligne-Mumford stack, the analogous of the operator $\tau_X$ is a map $\tau_\sX$ valued in a suitable extension of scalars of the cohomology of the inertia stack $I_\sX$ of $\sX$. Recall the definition of the inertia stack:
\begin{definition}[{\cite[8.1.17]{Ols16}}]
    Let $\sX$ be an algebraic stack. The \textit{inertia stack} $I_\sX$ is the fibered product of the diagram
    \[
    \begin{tikzcd}
     & \sX \arrow{d}{\Delta}\\
     \sX \arrow{r}{\Delta} & \sX \times \sX,
    \end{tikzcd}\]
    where $\Delta$ is the diagonal embedding. 
\end{definition}

\begin{remark}
\label{rmk:InertiaOfQuotient}
    More explicitly, the objects of $I_\sX$ are pairs $(x,g)$ where $x$ is an object of $\sX$ lying above a scheme $T$, and $g$ is an automorphism of $x$ in $\sX(T)$.
    
    Suppose $\sX = [Z/H]$ where $H$ is a finite group acting on a variety $Z$. Sections of $\sX$ are $H$-torsors equipped with an equivariant map to $Z$. Then, $I_\sX$ can be canonically identified with $[\hat{Z}/H]$, where $\hat{Z}=\{ (z,h)\in Z\times H \st h.z=z \}$. Therefore we have 
    \[ I_\sX \simeq \coprod_{(h)\in c(H)} [Z^h/C_G(h)],\]
    where $c(H)$ is the set of conjugacy classes of $H$, and $C_G(h)$ denotes the centralizer of a conjugacy class $(h)$.
\end{remark}


We denote by $\mu_{\infty}$ the subgroup of $\C^*$ containing all roots of unity, and define \(\Lambda \coloneqq \mathbb{Q}(\mu_\infty)\) to be the rational numbers adjoined $\mu_{\infty}$. As usual, for every $\Z$-module $A$ we denote by $A_\Lambda$ the tensor product $A \otimes_\Z \Lambda$.
Define a map 
\[ 
  \rho\colon K(I_\sX) \to K(I_\sX)_\Lambda
\]
as follows. A vector bundle $E$ over $I_\sX$ decomposes as a sum of eigenbundles
 $\oplus_{\zeta \in \mu_{\infty}} E^{(\zeta)}$ as in the proof of \cite[Th\'eor\`eme 3.15]{Toe99}\footnote{The decomposition, roughly, works as follows. A section $f\in I_{\sX}(T)$ for some scheme $T$ is the datum of a section $x\in \sX(T)$ with an automorphism $a$ of $x$. Then, the bundle $f^*E$ on $T$ is equipped with an action of $<a>$, which is diagonalizable by the tameness assumption. For $\zeta\in \mu_{\infty}$, one shows that the $\zeta$-eigenbundle of $f^*E$ can be written as $f^*E^{(\zeta)}$, for some vector bundle $E^{(\zeta)}$ on $I_\sX$.}. Then, let 
 \[  \rho(E)\coloneqq \sum\limits_{\zeta\in \mu_\infty} \zeta E^{(\zeta)}. \]

\begin{definition}
Let $\sX$ be a tame smooth Deligne-Mumford stack with quasi-projective coarse moduli space. 
Define the \textit{weighted Chern character}, \(\widetilde{\ch{}}\colon K(\sX)\to H^\ast(I_\sX)_\Lambda\), as the composition
\[ 
	K(\sX)\xrightarrow{\sigma^*} K(I_\sX) \xrightarrow{\rho}
  K(I_\sX)_\Lambda \xrightarrow{\ch{}} H^\ast(I_\sX)_\Lambda
\]
where $\sigma\colon I_\sX \to \sX$ is the projection (onto either factor) and $\ch{}$ is the usual Chern
character\footnote{The singular homology of a Deligne-Mumford stack coincides rationally with that of its coarse moduli space. Then, we can regard the Chern character as landing in the \textit{Chen-Ruan orbifold cohomology} of $\sX$, defined in \cite{CR04} as the singular homology of the coarse moduli space $\underline{I_\sX}$ of $I_\sX$. 
}.
\label{def_Chern_character}
\end{definition}

Next, we define the weighted Todd class of $\sX$. This is a modification of the usual Todd class of $I_\sX$. Let $N$ denote the normal bundle of the local immersion $\sigma\colon I_\sX \to \sX$, and define  
\[ \lambda_{-1}(N^\vee) \coloneqq \sum_i (-1)^i \bigwedge^i N^\vee \quad \in K(I_\sX). \]
The element $\rho(\lambda_{-1}(N^\vee))$ is invertible in $K(I_\sX)_\Lambda$ by \cite[Lemme 4.6]{Toe99}. Define the \textit{weighted Todd class} of $\sX$ as
\begin{equation}
    \label{eq:weightedTD}
\widetilde{\mathrm{Td}}(\sX) \coloneqq \frac{\mathrm{Td}(I_\sX)}{\ch{}(\rho(\lambda_{-1}(N^\vee)))} 
\end{equation}
and the \textit{To\"en map} $\tau_\sX \colon K(\sX) \to H^*(I_\sX)_{\Lambda}$ as
\[ \tau_\sX (E) \coloneqq \widetilde{\ch{}}(E) \cdot  \widetilde{\mathrm{Td}}(\sX). \]

To\"en's Riemann-Roch theorem for stacks asserts that $\tau_\sX$ behaves functorially with respect to proper push forwards:

\begin{thm}[{\cite[Th\'eor\`eme 4.10]{Toe99}}]
Let $f\colon \sX \to \sY$ be a proper morphism of smooth Deligne-Mumford stacks with quasi-projective coarse moduli spaces. Then for all \(E \in K(\sX)\) we have
\[ f_*(\tau_\sX(E)) = \tau_{\sY}(f_*E).\]
Moreover, if $f\colon \sX \to \mathrm{pt}$, we obtain
    \begin{equation}
        \label{eq:THRR}
    \chi(E) = \int\limits_{I_\sX}\tau_\sX(E).
    \end{equation}
    \label{thm:tgrr}
\end{thm}
Expanding the expression \eqref{eq:THRR} we have 
\begin{equation}
    \label{eq:correctionterm}
 \chi(E) = \int\limits_{I_\sX}\widetilde{\ch{}}(E).\widetilde{\mathrm{Td}}(\sX) =  \int\limits_{\sX}\widetilde{\ch{}}(E).\widetilde{\mathrm{Td}}(\sX) + \delta(E),
\end{equation}
with $\delta(E)\coloneqq \int\limits_{I_\sX\setminus
\sX}\widetilde{\ch{}}(E).\widetilde{\mathrm{Td}}(\sX)\in \Q$ the
correction term.


\section{Riemann-Roch for kleinian orbisurfaces}
\label{sec:orbisurf}


\subsection{Kleinian orbisurfaces}

Our object of study is the stacky resolution of singularities of \(ADE\) type (also known as Kleinian singularities).

\begin{definition}
  An \textit{orbisurface} is a smooth, proper, and tame Deligne-Mumford surface $\sS$ over an algebraically closed field $\mathbf{k}$ with projective coarse moduli and isolated stacky locus. 
  \label{def:orbisurface}
\end{definition}

For any orbisurface \(\sS\), the stacky locus is a finite union of residual gerbes corresponding to finitely many \(\mathbf{k}\)-points \(p_i\in\sS(\mathbf{k})\), i.e.
\[
	\mathrm{Stack}(\sS) = \coprod_{i=1}^rBG_i
\]
where \(G_i = \mathrm{stab}(p_i)\) is a finite subgroup of $\GL_2$.

\begin{definition}
An orbisurface is \textit{Kleinian} if each \(G_i\) is a subgroup of \(\mathrm{SL}_2\).
	\label{def_Klein_orbisurface}
\end{definition}

\begin{example} Let \(S\) be a surface with tame Kleinian singularities\footnote{The $ADE$ classification holds in characteristic $p>0$ as well, as long as the orders of the stabilizers is coprime with $p$. In this setting, quotients of $\mathbb{A}^2$ by finite subgroups of $\SL_2$ classify $F$-rational Gorenstein surface singularities \cite[\S 3]{Has15}.}. Then there exists a Kleinian
orbisurface \(\sS^{can}\) and a map \(\pi\colon\sS^{can}\to S\) such that:
\begin{itemize}
\item the restriction $  \sS^{can}\setminus\pi^{-1}(\mathrm{Sing}(S))\to S\setminus\mathrm{Sing}(S)$ is an isomorphism;
\item $\pi$ is universal among all dominant, codimension
preserving maps to \(S\).
\end{itemize}
The stack \(\sS^{can}\) is called the \textit{canonical stack}
associated with the surface \(S\), see \cite{FMN10}.
\end{example}

We will compute a formula for the correction term $\delta(E)$ appearing in \eqref{eq:correctionterm} in the case of a Kleinian orbisurface. 
Since \(\delta(E)\) is computed at each residual gerbe independently, we may and will assume that \(\sS\) has a single stacky point, \(p\), with residual gerbe \(\iota\colon BG\hookrightarrow \sS\). 
We will see that the correction terms involve coefficients determined solely by the natural action of \(G\) on the tangent space \(T_p\sS\). For any subgroup $G$ of $\GL_2$, we denote by $V\simeq \mathbb{A}^2$ the natural representation. 

\subsection{The weighted Todd class}

Let $\sS$ be a Kleinian orbisurface with a single stacky point $p$ with stabilizer $G$.
Arguing as in Remark \ref{rmk:InertiaOfQuotient}, we see that the inertia stack of $\sS$ is 
\[ 
  I_\sS = \sS \sqcup (I_{BG} \setminus BG).
\]
Here 
\[ 
  I_{BG} \setminus BG = \bigsqcup\limits_{(g)\neq (1)}BC_G(g),
\]
where the union is taken over all conjugacy classes $(g)$ of non-trivial elements $g\in G$.
Fix one of the components $BC_G(g)$. Its normal bundle in $\sS$ is identified with $T_p\sS = V^\vee$. Then the class $\lambda_{-1}(N^\vee)$ restricted to $BC_G(g)$ is
\[\lambda_{-1}(N^\vee)_{|BC_G(g)} =[\mathbf{1}] - [V] + [\wedge^2V] = 2[\mathbf{1}] -[V]\]
in $K(BC_G(g))$ (which is free, abelian, and generated by irreducible representations of $C_G(g)$). The element $g$ acts diagonally on $V$, with eigenvalues some roots of unity $\xi_g$ and $\xi_g^{-1}$. Thus, 
\[ \ch{}(\rho(\lambda_{-1}(N^\vee)_{|BC_G(g)})) = 2-\xi_g-\xi_g^{-1}=2-\chi_V(g) \quad \in \Q(\mu_\infty),\]
where \(\chi_V(g) = \xi_g+\xi_g^{-1}\) is the character of $V$ evaluated at $g$, i.e. the trace of $g$ acting through the representation $V$.
Using \eqref{eq:weightedTD} we obtain:
\[\int\limits_{BC_G(g)} \widetilde{\mathrm{Td}}_\sS =
  \int\limits_{BC_G(g)} \frac{1}{2-\chi_V(g)}= \frac{1}{\abs{C_G(g)}} \cdot \frac{1}{2-\chi_V(g)}. 
\]
Integrating over the twisted sector:
\begin{equation}
    \label{eq:TDon TwistedSector}
 \delta(\sO_\sS) = \int\limits_{I_{BG}\setminus BG} \widetilde{\mathrm{Td}}_\sS =
  \sum\limits_{(g)\neq (1)}\frac{1}{\abs{C_G(g)}} \cdot \frac{1}{2-\chi_V(g)} 
\end{equation}

\begin{remark}
If \(\sS\to S\) is the projection to the coarse moduli space, then \(S\) has an ADE singularity at the image of the stacky point. The integral \eqref{eq:TDon TwistedSector} is computed in \cite{CT19} to be 
\[
  \delta(\sO_\sS) = \frac{1}{12}\left(\chi_{top}(C_{red})-\frac{1}{\abs{G}}\right),
\]
where $C$ is the fundamental cycle of the minimal resolution of the singularity.
\end{remark}

\subsection{Riemann-Roch coefficients}

In this section, we write an expression of the term $\delta(E)$ appearing in \eqref{eq:correctionterm} in terms of the wieghted Chern character of $E$ and of the character table of the stabilizer group $G$. 

The Grothendieck group of \(BG\) is free, Abelian and generated by the
irreducible representations of $G$ $\{ \rho_i \,|\,i=0,...,M \}$. For any
perfect complex of sheaves \(E\) on \(\sS\), its derived fiber is a formal linear combination
\[
    [{\dL}\iota^\ast E] = \sum_{i=0}^{M} a_i\rho_i \in K(BG).
\]
On the component $BC_G(g)$, the element $g$ acts on $\rho_i$ with eigenvalues denoted $\zeta_i^{(l)}$,
to which correspond eigenspaces $\rho_i^{(l)}$ (the action is diagonalizable by the tameness assumption). Therefore, $\dL\iota^*E$
decomposes on $BC_G(g)$ into weighted eigenbundles as 
\[\sum\limits_{i=0}^M
\sum\limits_{l=1}^{\dim \rho_i} a_i \zeta_i^{(l)} \rho_i^{(l)}.\]
Denote by $\chi_i \coloneqq \chi_{\rho_i} = \mathrm{Tr}\circ \rho_i$ the character of the repesentation $\rho_i$.

\begin{definition}
    For each \(i=0,\ldots,M\) set
	\begin{equation}\label{eq_formula_for_Ti}
		T_i \coloneqq  \sum\limits_{(g)\neq (1)} \frac{\chi_i(g)}{\abs{C_G(g)}(2-\chi_V(g))}.
	\end{equation}
 	We call the \(T_i\) the \textit{Riemann-Roch coefficients} of \(G\).
\end{definition}

\begin{proposition}
	The correction term for a complex of sheaves \(E\) on \(\sS\) can be written as
	\[
		\delta(E) = \sum_{i=0}^Ma_iT_i.
	\]
	In particular, it only depends on the ranks of eigenbundles of ${\dL}\iota^\ast E$ and the coefficients \(T_i\). The latter only depend on the character table of \(G\).
	\label{thm_RR_coeff}
\end{proposition}

\begin{proof}
The weighted Chern character of $\dL\iota^*E_{|BC_G(g)}$ is given by 
\begin{equation*}\label{eq_weighted_Chern_char}
  \widetilde{\ch{}}(\dL\iota^*E_{|BC_G(g)}) = \sum\limits_{i=0}^M \sum_{l=1}^{r_i} a_i \zeta_i^{(l)} =  \sum_{i=0}^M a_i \chi_i(g).
\end{equation*}

Thus we have
\begin{equation}\label{eq_formula_for_delta}
\delta(E)= 
\int\limits_{I_\sS\setminus     \sS}\widetilde{\ch{}}(E).\widetilde{\mathrm{Td}}_\sS=
     \sum\limits_{(g)\neq (1)} \frac{1}{\abs{C_G(g)}}\cdot\frac{\sum_ia_i\chi_i(g)}{2-\chi_V(g)}=
\sum_i a_i T_i,
\end{equation}
with 
\[
  T_i\coloneqq  \sum\limits_{(g)\neq (1)} \frac{\chi_i(g)}{\abs{C_G(g)}(2-\chi_V(g))}. \qedhere
\]
\end{proof}

\section{Computation of Riemann-Roch Coefficients}
\label{sec_Correction_terms}

Now we obtain formulae for the correction term \eqref{eq_formula_for_delta}, by explicitly computing the corresponding Riemann-Roch coefficients $T_i$ for all Kleinian singularities. This extends the computation done in \cite{Lie11} for singularities of type $A$, and completes the one started in \cite{CT19}, where the authors only compute $\delta(\sO_\sS)$. The computations do not depend on the characteristic of the base field, and neither do the results.

\subsection{Singularities of type \texorpdfstring{$A$}{A}} In this case, the coefficients $T_j$ are computed by Lieblich in \cite[Sec. 3.3.2]{Lie11}, who gives an explicit formula for $\delta(F)$. We recall his result here. We'll make use of the following Lemma:
\begin{lemma}[{\cite[Lemma 3.3.2.1]{Lie11}}]\label{lem_Lieblich}
Let $\zeta$ be a primitive $P$-th root of unity and $j\leq P$ a non-negative integer. Then
\[ \sum\limits_{k=1}^{P-1}\frac{\zeta^{kj}}{2-\zeta^k -\zeta^{-k}}=\frac{j(j-P)}{2} + \frac{P^2-1}{12}. \]
\end{lemma}

The \(K\)-theory of \(B\mu_N\) is free Abelian of rank $N$ with $\{ \chi^j \,|\,j=0,...,N-1 \}$ 
as a basis. For any perfect complex of sheaves \(\sF\) on \(\sS\), we
have
\[
    [{\bf L}\iota^\ast\sF] = \sum_{j=0}^{N-1} a_j\chi^j.
\]
Define a function $f\colon \Z/N\Z \to \Q$ by the formula
\[ f(\overline{x})=\frac{x(x-N)}{2} + \frac{N^2-1}{12}. \]

Let $\zeta\in G$ be a primitive $N$-th root of unity. It acts on $\chi^j$ with multiplication by $\zeta^j$. Then, equation \eqref{eq_formula_for_Ti} reads
\[ T_j = \frac 1N \sum\limits_{k=1}^{N-1}\frac{\zeta^{kj}}{2-\zeta^k -\zeta^{-k}} =\frac 1N f(j) \]
as a consequence of Lemma \ref{lem_Lieblich}.

\subsection{Singularities of type \texorpdfstring{$D$}{D}: Binary dihedral groups} In this case, the group acting is the binary dihedral group $G=\mathrm{Dic}_n$, it has order $4n$ and it gives rise to a singularity of type $D_{n+2}$ with $n\geq 2$. We can present it as 
\[ \mathrm{Dic}_n=\pair{a,x \,\lvert\, a^{2n}=I, x^2=a^n, x^{-1}ax=a^{-1}}. \]
The center of $G$ is cyclic of order 2, generated by $x^2=a^n$. The quotient of $G$ by its center is the dihedral group $\mathrm{Dih}_n$ with $2n$ elements. 

For the representation theory of $\mathrm{Dic}_n$, we point the reader to \cite[\S 13]{IN99} or to \cite[\S 7.1]{Cox74}. The group $G$ has $n+3$ conjugacy classes, grouped by cardinality as: 
\begin{align*}
  \set{I},  \set{-I=x^2=a^n}\\
  \set{a,a^{-1}}, \set{a^2,a^{-2}}, ... \set{a^{n-1},a^{n+1}},\\
  \set{xa,xa^3,...,xa^{2n-2}}, \set{x,xa^2,...,xa^{2n-1}}.
\end{align*}
The corresponding centralizers have cardinality $4n$, $2n$ and $4$.

There are 4 one-dimensional representations. There are several two-dimensional representations, called \textit{dihedral}, induced by $G\to \mathrm{Dih}_n$ there are $\frac{n-1}{2}$ of these if $n$ is odd, and $\frac{n-2}{2}$ if $n$ is even. The $l$-th dihedral representation is given by the assignment
\[ a\mapsto\begin{pmatrix} e^{\frac{2l\pi i}{n}} & 0\\ 0 & e^{-\frac{2l\pi i}{n}} \end{pmatrix}; \qquad x=\begin{pmatrix} 0 & 1\\ 1 & 0 \end{pmatrix}. \]

 The remaining representations are called of \textit{quaternionic type} since they are induced by an inclusion $G\subset \SL(2,\C)$. They are also two-dimensional, there are $\frac{n-1}{2}$ if $n$ is odd, or $\frac n2$ if $n$ is even. The $l$-th quaternionic representation is given by  
 \[ a\mapsto\begin{pmatrix} e^{\frac{l\pi i}{n}} & 0\\ 0 & e^{-\frac{l\pi i}{n}} \end{pmatrix}; \qquad x=\begin{pmatrix} 0 & -1\\ 1 & 0 \end{pmatrix}. \]
We assume that $G$ acts on $\C^2$ via the first ($l=1$) quaternionic representation, and denote by $V$ this representation. 

Let $T_\rho$ be the Riemann-Roch coefficient corresponding to a representation $\rho$. For $G=\mathrm{Dic}_n$, we have
\[ T_\rho = \frac{\chi_\rho(-I)}{16n} + \frac 18(\chi_\rho(x)+\chi_\rho(xa)) + \frac{1}{2n}\sum_{k=1}^{n-1}\frac{\chi_\rho(a^k)}{2-\chi_V(a^k)}. \]

First, we compute coefficients for the 4 one-dimensional representations. Let $\rho_x$ be the representation where $x$ acts by -1 and $a$ acts trivially, similarly for $\rho_a$, and write $\rho_{x,a}=\rho_x\otimes\rho_a$. Let $\zeta\coloneqq e^{-\frac{\pi i}{n}}$. By applying Lemma \ref{lem_Lieblich} we obtain
\begin{align}
T_{\rho_0}&= \frac{1}{16n} + \frac 14 + \frac{n^2-1}{12n}\\
        T_{\rho_x}&= \frac{1}{16n} - \frac 14 + \frac{n^2-1}{12n}.    \end{align}
The remaining coefficient is 
\begin{equation}\label{eq_express T_a}
T_{\rho_a}= T_{\rho_{x,a}}= \frac{1}{16n} + \frac{1}{2n}\sum_{k=1}^{n-1}\frac{(-1)^k}{2-\chi_V(a^k)}=-\frac{1}{16n} + \frac{1}{2n}\sum_{k=1}^{n-1}\frac{\zeta^{nk}}{2-\zeta^k-\zeta^{-k}}.
\end{equation}
Now we compute the the sum in \eqref{eq_express T_a}: substituting $h=2n-k$ we have
\begin{equation}\label{eq_substitute_hk}
    \begin{split}
    \sum_{k=1}^{n-1}\frac{\zeta^{-kl}}{\left(2-\zeta^k-\zeta^{-k}\right)}= \sum_{h=n+1}^{2n-1}\frac{\zeta^{(h-2n)l}}{\left(2-\zeta^{2n-h}-\zeta^{h-2n}\right)}= \sum_{h=n+1}^{2n-1}\frac{\zeta^{hl}}{\left(2-\zeta^{-h}-\zeta^{h}\right)}.
    \end{split}
\end{equation}
Using the \eqref{eq_substitute_hk} with $l=n$, write
\begin{equation*}
     \sum_{k=1}^{2n-1}\frac{\zeta^{nk}}{\left(2-\zeta^k-\zeta^{-k}\right)} = \sum_{k=1}^{n-1}\frac{\zeta^{nk}}{\left(2-\zeta^k-\zeta^{-k}\right)} + \frac{(-1)^n}{4} + \sum_{k=1}^{n-1}\frac{\zeta^{-nk}}{\left(2-\zeta^k-\zeta^{-k}\right)}.
     \end{equation*}
whence, applying Lemma \ref{lem_Lieblich} with $P=2n$ and $j=n$,
\begin{equation*}
\begin{split}
\sum_{k=1}^{n-1}\frac{\zeta^{nk}}{\left(2-\zeta^k-\zeta^{-k}\right)}= \frac 12 \left[     \sum_{k=1}^{2n-1}\frac{\zeta^{nk}}{\left(2-\zeta^k-\zeta^{-k}\right)} - \frac{(-1)^n}{4} \right] = \\
\frac 12 \left[    -\frac{n^2}{2} + \frac{4n^2-1}{12} - \frac{(-1)^n}{4} \right].
\end{split}
     \end{equation*}
    
If $\sigma$ is the $l$-th quaternionic representation, we have 
\begin{equation}
 T_\sigma = -\frac{1}{8n} +  \frac{1}{2n}\sum_{k=1}^{n-1}\frac{\left(\zeta^{kl}+\zeta^{-kl}\right)}{\left(2-\zeta^k-\zeta^{-k}\right)}. 
\end{equation}
Using once again the substitution \eqref{eq_substitute_hk}, we obtain:
\begin{equation*}
    \begin{split}
     \sum_{k=1}^{n-1}\frac{\left(\zeta^{kl}+\zeta^{-kl}\right)}{\left(2-\zeta^k-\zeta^{-k}\right)} = \sum_{k=1}^{n-1}\frac{\zeta^{kl}}{\left(2-\zeta^k-\zeta^{-k}\right)} + \sum_{k=1}^{n-1}\frac{\zeta^{-kl}}{\left(2-\zeta^k-\zeta^{-k}\right)}=\\
     \sum_{k=1}^{n-1}\frac{\zeta^{kl}}{\left(2-\zeta^k-\zeta^{-k}\right)} + \sum_{h=n+1}^{2n-1}\frac{\zeta^{hl}}{\left(2-\zeta^{-h}-\zeta^{h}\right)}= \sum_{k=1}^{2n-1}\frac{\zeta^{kl}}{\left(2-\zeta^{-k}-\zeta^{k}\right)} - \frac{(-1)^l}{4}.
    \end{split}
\end{equation*}
Applying Lemma \ref{lem_Lieblich} with $P=2n$ and $j=l$ then yields
\begin{equation*}
    \begin{split}
     \sum_{k=1}^{n-1}\frac{\left(\zeta^{kl}+\zeta^{-kl}\right)}{\left(2-\zeta^k-\zeta^{-k}\right)} = \frac{l(l-2n)}{2} + \frac{(2n)^2-1}{12}- \frac{(-1)^l}{4},
    \end{split}
\end{equation*}
and finally
\begin{equation}
T_\sigma = -\frac{1}{8n} +  \frac{1}{2n}\left( -\frac {(-1)^l}4 + \frac{l(l-2n)}{2} + \frac{(2n)^2-1}{12} \right). 
\end{equation}

If $\tau$ is a dihedral representation, then
\begin{equation}
T_\tau = -\frac{1}{8n} +  \frac{1}{2n}\sum_{k=1}^{n-1}\frac{\left(\zeta^{2kl}+\zeta^{-2kl}\right)}{\left(2-\zeta^{k}-\zeta^{-k}\right)} 
\end{equation}
and we may repeat the argument above applying Lemma \ref{lem_Lieblich} with $P=2n$ and $j=2l$:
\begin{equation}
\sum_{k=1}^{n-1}\frac{\left(\zeta^{2kl}+\zeta^{-2kl}\right)}{\left(2-\zeta^{k}-\zeta^{-k}\right)}=
\left( -\frac {1}4 + 2l(l-n) + \frac{(2n)^2-1}{12} \right).
\end{equation}

\subsubsection{Example: the $D_4$ singularity} The binary dihedral group $G=\mathrm{Dic}_n$ with $n=2$ has order $4n=8$. We present it as the group of matrices generated by 
\[ a=\begin{pmatrix} i & 0\\ 0 & -i \end{pmatrix}; \qquad x=\begin{pmatrix} 0 & -1\\ 1 & 0 \end{pmatrix}. \]

As above, let $\rho_0, \rho_a, \rho_x$ and $\rho_{xa}$ denote the one-dimensional representations. The two dimensional irreducible representation is denoted $V$.

For a sheaf $F$ we can write $[\dL i^*F]=\sum_{\rho}a_\rho \rho$, then the correction term is 
\[ \delta(F)= \sum_{\rho} a_\rho T_\rho \]
with $T_{\mathbbm 1}=\frac{13}{32}$, $T_{\rho_a}=T_{\rho_x}=T_{\rho_{xa}}=-\frac{3}{32}$ and $T_V=-\frac{2}{32}$.

As an example, we explicitly check Lemma 2.9 of \cite{LR20} in a few cases. It states that 
\[\delta(\sO_p\otimes \rho)=\begin{cases} 1-\frac{1}{|G|} &\mbox{ if }\rho=\rho_0\\
-\frac{\dim \rho}{|G|} &\mbox{ else}.
 \end{cases}\]

 Let $F=\sO_p$. Then the equivariant Koszul complex
 \[ \sO\otimes \Lambda^2V \to \sO \otimes V \to \sO \to \sO_p \]
 shows $[\dL i^*\sO_p]= 2\rho_0-V$. Plugging this in, we get 
\[ \delta(\sO_p)=2\cdot\frac{13}{32} - \frac{-2}{32}= \frac 78. \]

Now let $F=\sO_p\otimes V$. Observe that $V\otimes V = \oplus_{j=0}^3 \chi^j$, whence $[\dL i^*\sO_p]= 2V-(\sum \chi^j)$, i.e. $a_0=...=a_3=-1$ and $a_4=2$. The same computation as above yields $\delta(\sO_p\otimes V)=-\frac 14$.

\subsection{Singularities of type \texorpdfstring{$E$}{E}}
The groups giving rise to singularities of type $E$ are the binary tetrahedral ($E_6$), binary octahedral ($E_7$) and binary icosahedral group ($E_8$). We follow the notation of \cite[\S 14-16]{IN99}. These groups are constructed as follows. Set
\[
	\sigma = \begin{bmatrix}
		i & 0 \\
		0 & -i
	\end{bmatrix},\ 
	\tau = \begin{bmatrix}
	0 & 1 \\
	-1 & 0
	\end{bmatrix},\ 
	\mu = \frac{1}{\sqrt{2}}\begin{bmatrix}
	\varepsilon^7 & \varepsilon^7 \\
	\varepsilon^5 & \varepsilon \\
	\end{bmatrix}
\]
where \(\varepsilon = e^{\pi i/4}\). Then the binary tetrahedral group is the group \(2T = \langle \sigma, \tau,\mu\rangle\). Additionally, set
\[
	\kappa = \begin{bmatrix}
	\varepsilon & 0 \\
	0 &\varepsilon^7
	\end{bmatrix}.
\]
Then the binary octahedral group is \(2O = \langle 2T,\kappa\rangle\).

For the binary icosahedral group \(2I\) we set
\[
	\sigma = -\begin{bmatrix}
	\varepsilon^3 & 0 \\
	0 & \varepsilon^2 
	\end{bmatrix},\ 
	\tau = \frac{1}{\sqrt{5}}\begin{bmatrix}
	-(\varepsilon-\varepsilon^4) & \varepsilon^2-\varepsilon^3 \\
	\varepsilon2-\varepsilon^3 & \varepsilon-\varepsilon^4
	\end{bmatrix}
\]
where \(\varepsilon = e^{2\pi i/5}\). Then \(2I = \langle \sigma,\tau\rangle\). 

\subsubsection{The \(E_6\) singularity} The binary tetrahedral group $2T$ is of order 24. Its character table is in Table \ref{tab_Char_e6}, where \(\omega = \frac{-1+\sqrt{3}i}{2}\).
\begin{table}[h!]
\begin{center}
\begin{tabular}{c|ccccccc}
 & $1$ & $-1$ & $\tau$ & $\mu$ & $\mu^2$ & $\mu^4$ & $\mu^5$ \\
 \hline
$\abs{C_G(g)}$ & 24 & 24 & 4 & 6 & 6 & 6 & 6\\
\hline
$\rho_0$& $1$ & $1$ & $1$ & $1$ & $1$ & $1$ & $1$ \\
$\rho_2$ & $2$ & $-2$ & $0$ & $1$ & $-1$ & $-1$ & $1$ \\
$\rho_3$ & $3$ & $3$ & $-1$ & $0$ & $0$ & $0$ & $0$ \\
$\rho_2'$ & $2$ & $-2$ & $0$ & $\omega^2$ & $-\omega$ & $-\omega^2$ & $\omega$\\
$\rho_1'$ & $1$ & $1$ & $1$ & $\omega^2$ & $\omega$ & $\omega^2$ & $\omega$ \\
$\rho_2''$ & $2$ & $-2$ & $0$ & $\omega$ & $-\omega^2$ & $-\omega$ & $\omega^2$\\
$\rho_1''$  & $1$ & $1$ & $1$ & $\omega$ & $\omega^2$ & $\omega$ & $\omega^2$ \\ [1 ex]
\end{tabular}
\end{center}
\caption{Character Table for \(2T\)}
\label{tab_Char_e6}
\end{table}
Here \(\rho_2\) is the natural representation. The Riemann-Roch coefficients are computed as follows:
\begin{equation}
T_i =\frac{\chi_i(-1)}{96} + \frac{\chi_i(\tau)}{8} + \frac{\chi_i(\mu)}{6} + \frac{\chi_i(\mu^2)}{18} + \frac{\chi_i(\mu^4)}{18} + \frac{\chi_i(\mu^5)}{6}
\label{eqn_RR_e6}
\end{equation}

From \eqref{eqn_RR_e6} we can compute the Riemann-Roch coefficients which we arrange in Table \ref{tab_RR_e6}.
\bgroup
\def\arraystretch{1.3}
\begin{table}[h!]
\begin{center}
\begin{tabular}{c|cccccccc}
 & $\rho_0$ & $\rho_2$ & $\rho_3$ & $\rho_2'$ & $\rho_1'$ & $\rho_2''$ & $\rho_1''$ \\ \hline
$T_i$ & $\frac{167}{288}$ & $\frac{29}{144}$ & $-\frac{3}{32}$ & $-\frac{19}{144}$ & $-\frac{25}{288}$ & $-\frac{19}{144}$ & $-\frac{25}{288}$  \\ [1ex]
\end{tabular}
\end{center}
\caption{Riemann-Roch Coefficients for \(2T\)}
\label{tab_RR_e6}
\end{table}
\egroup

\subsubsection{The \(E_7\) singularity} 
The binary octahedral group \(2O\) is of order 48. Its character table is in Table \ref{tab_Char_e7}.

\begin{table}[h!]
\begin{center}
\begin{tabular}{c|cccccccc}
 & $1$ & $-1$ & $\mu$ & $\mu^2$ & $\tau$ & $\kappa$ & $\tau\kappa$ & $\kappa^3$\\
 \hline
$\abs{C_G(g)}$ & 48 & 48 & 6 & 6 & 8 & 8 & 4 & 8 \\
\hline
$\rho_0$& $1$ & $1$ & $1$ & $1$ & $1$ & $1$ & $1$ & $1$ \\
$\rho_2$ & $2$ & $-2$ & $1$ & $-1$ & $0$ & $\sqrt{2}$ & $0$ & $-\sqrt{2}$ \\
$\rho_3$ & $3$ & $3$ & $0$ & $0$ & $-1$ & $1$ & $-1$ & $1$ \\
$\rho_4$ & $4$ & $-4$ & $-1$ & $1$ & $0$ & $0$ & $0$ & $0$ \\
$\rho_3'$ & $3$ & $3$ & $0$ & $0$ & $-1$ & $-1$ & $1$ & $-1$ \\
$\rho_2'$ & $2$ & $-2$ & $1$ & $-1$ & $0$ & $-\sqrt{2}$ & $0$ & $\sqrt{2}$ \\
$\rho_1'$ & $1$ & $1$ & $1$ & $1$ & $1$ & $-1$ & $-1$ & $-1$ \\
$\rho_2''$ & $2$ & $2$ & $-1$ & $-1$ & $2$ & $0$ & $0$ & $0$ \\ [1ex]
\end{tabular}
\end{center}
\caption{Character Table for \(2O\)}
\label{tab_Char_e7}
\end{table}
Here \(\rho_2\) is the natural representation. The Riemann-Roch coefficients are computed as follows:
\begin{equation}
T_i = \frac{\chi_i(-1)}{192} + \frac{\chi_i(\mu)}{6} + \frac{\chi_i(\mu^2)}{18} + \frac{\chi_i(\tau)}{16} + \frac{\chi_i(\kappa)}{16-8\sqrt{2}} + \frac{\chi_i(\tau\kappa)}{8} + \frac{\chi_i(\kappa^3)}{16+8\sqrt{2}}
\label{eqn_RR_e7}
\end{equation}

We arrange them in Table \ref{tab_RR_e7}.
\bgroup
\def\arraystretch{1.3}
\begin{table}[h!]
\begin{center}
\begin{tabular}{c|cccccccc}
 & $\rho_0$ & $\rho_2$ & $\rho_3$ & $\rho_4$ & $\rho_3'$ & $\rho_2'$ & $\rho_1'$ & $\rho_2''$\\ \hline
$T_i$ & $\frac{383}{576}$ & $\frac{101}{288}$ & $\frac{5}{64}$ & $-\frac{19}{144}$ & $-\frac{11}{64}$ & $-\frac{43}{288}$ & $-\frac{49}{576}$ & $-\frac{26}{288}$ \\ [1ex]
\end{tabular}
\end{center}
\caption{Riemann-Roch Coefficients for \(2O\)}
\label{tab_RR_e7}
\end{table}
\egroup

\subsubsection{The \(E_8\) singularity} 

The binary icosahedral group \(2I\) is of order 120. Its character table is in Table \ref{tab_Char_e8}, where \(\mu^\pm = \frac{1\pm\sqrt{5}}{2}\).

\begin{table}[h!]
\begin{center}
\begin{tabular}{c|ccccccccc}
 & $1$ & $-1$ & $\sigma$ & $\sigma^2$ & $\sigma^3$ & $\sigma^4$ & $\tau$ & $\sigma^2\tau$ & $\sigma^7\tau$\\
 \hline
$\abs{C_G(g)}$ & 120 & 120 & 10 & 10 & 10 & 10 & 4 & 6 & 6 \\
\hline
$\rho_0$& $1$ & $1$ & $1$ & $1$ & $1$ & $1$ & $1$ & $1$ & $1$ \\
$\rho_2$ & $2$ & $-2$ & $\mu^+$ & $-\mu^-$ & $\mu^-$ & $-\mu^+$ & $0$ & $-1$ & $1$ \\
$\rho_3$ & $3$ & $3$ & $\mu^+$ & $\mu^-$ & $\mu^-$ & $\mu^+$ & $-1$ & $0$ & $0$ \\
$\rho_4$ & $4$ & $-4$ & $1$ & $-1$ & $1$ & $-1$ & $0$ & $1$ & $-1$\\
$\rho_5$ & $5$ & $5$ & $0$ & $0$ & $0$ & $0$ & $1$ & $-1$ & $-1$ \\
$\rho_6$ & $6$ & $-6$ & $-1$ & $1$ & $-1$ & $1$ & $0$ & $0$ & $0$ \\
$\rho_4'$ & $4$ & $4$ & $-1$ & $-1$ & $-1$ & $-1$ & $0$ & $1$ & $1$\\
$\rho_2'$ & $2$ & $-2$ & $\mu^-$ & $-\mu^+$ & $\mu^+$ & $-\mu^-$ & $0$ & $-1$ & $1$ \\
$\rho_3''$ & $3$ & $3$ & $\mu^-$ & $\mu^+$ & $\mu^+$ & $\mu^-$ & $-1$ & $0$ & $0$ \\ [1ex]
\end{tabular}
\end{center}
\caption{Character Table for \(2I\)}
\label{tab_Char_e8}
\end{table}

Here \(\rho_2\) is the natural representation. The Riemann-Roch coefficients are computed as follows:
\begin{equation}
\begin{split}
T_i =  &\frac{\chi_i(-1)}{480}+\frac{\chi_i(\sigma)}{20-10\mu^+} + \frac{\chi_i(\sigma^2)}{20+10\mu^-} + \frac{\chi_i(\sigma^3)}{20-10\mu^-} + \\
& + \frac{\chi_i(\sigma^4)}{20+10\mu^+} + \frac{\chi_i(\tau)}{8} + \frac{\chi_i(\sigma^2\tau)}{18} + \frac{\chi_i(\sigma^7\tau)}{6}.
\end{split}
\label{eqn_RR_e8}
\end{equation}
We arrange them in Table \ref{tab_RR_e8}.
\bgroup
\def\arraystretch{1.3}
\begin{table}[h!]
\begin{center}
\begin{tabular}{c|ccccccccc}
 & $\rho_0$ & $\rho_2$ & $\rho_3$ & $\rho_4$ & $\rho_5$ & $\rho_6$ & $\rho_4'$ & $\rho_2'$ & $\rho_3''$\\ \hline
$T_i$ & $\frac{1079}{1440}$ & $\frac{73}{144}$ & $\frac{9}{32}$ & $\frac{29}{360}$ & $-\frac{25}{288}$ & $-\frac{17}{80}$ & $-\frac{61}{360}$ & $-\frac{67}{720}$ & $-\frac{19}{160}$\\ [1ex]
\end{tabular}
\end{center}
\caption{Riemann-Roch Coefficients for \(2O\)}
\label{tab_RR_e8}
\end{table}
\egroup


\providecommand{\bysame}{\leavevmode\hbox to3em{\hrulefill}\thinspace}
\providecommand{\MR}{\relax\ifhmode\unskip\space\fi MR }
\providecommand{\MRhref}[2]{%
  \href{http://www.ams.org/mathscinet-getitem?mr=#1}{#2}
}
\providecommand{\href}[2]{#2}

\end{document}